\documentclass[letterpaper, 10 pt, conference]{ieeeconf} 
\usepackage{amsmath}
\usepackage{amsfonts}
\usepackage{xcolor}
\usepackage{graphicx}
\usepackage[percent]{overpic}

\newtheorem{proposition}{Proposition}

\IEEEoverridecommandlockouts                              
\title{\LARGE \bf
Safety-constrained position tracking of a flexible transmission system*
}

\author{Anthony Hastir$^{1}$ and Timo Reis$^{2}$
\thanks{*This research was conducted with the financial support of the Deutsche Forschungsgemeinschaft (DFG, German Research Foundation). T.~Reis is supported by the DFG - Project-ID 362536361, \textit{Adaptive
control of coupled rigid and flexible multibody systems with port-Hamiltonian structure}. All the authors acknowledge funding by the DFG, Germany
- Project-ID 531152215 - CRC 1701.}
\thanks{$^{1}$Anthony Hastir is with the Department of Mathematics and naXys,
        University of Namur, 5000 Namur, Belgium
        {\tt\small anthony.hastir@unamur.be}}%
\thanks{$^{2}$Timo Reis is with the Institut f\"ur Mathematik,
Technische Universit\"at Ilmenau, Weimarer Stra\ss e 25,
98693 Ilmenau, Germany
        {\tt\small timo.reis@tu-ilmenau.de}}%
}

\begin{document}

\maketitle
\thispagestyle{empty}
\pagestyle{empty}

\begin{abstract}
We consider safety-constrained position tracking for a flexible
transmission system consisting of a torsional shaft driven by a motor
and connected to a gravitationally loaded inertia. The coupled
PDE--ODE dynamics admit a semilinear formulation in which the linear
input--output system is impedance passive and the gravitational torque
enters as a bounded, globally Lipschitz perturbation. We establish an
energy dissipation inequality for the nonlinear model and show that a
position-funnel controller with supervised damping admits a unique
global closed-loop trajectory. In particular, the motor-angle tracking
error remains strictly inside a prescribed time-varying funnel, which
can be chosen to enforce safety bounds on the motor angle.
\end{abstract}

\section{INTRODUCTION}
We study position tracking for a flexible transmission system
consisting of a torsional shaft with a motor at one endpoint and a
rigid load at the other. The control input is the motor torque,
while gravity exerts a nonlinear torque on the load. The controlled
output is the motor angle. The configuration is illustrated in
Figure~\ref{fig:flexible_transmission}.

\begin{figure}[b]
    \centering
    \begin{overpic}[
        width=\linewidth,
    ]{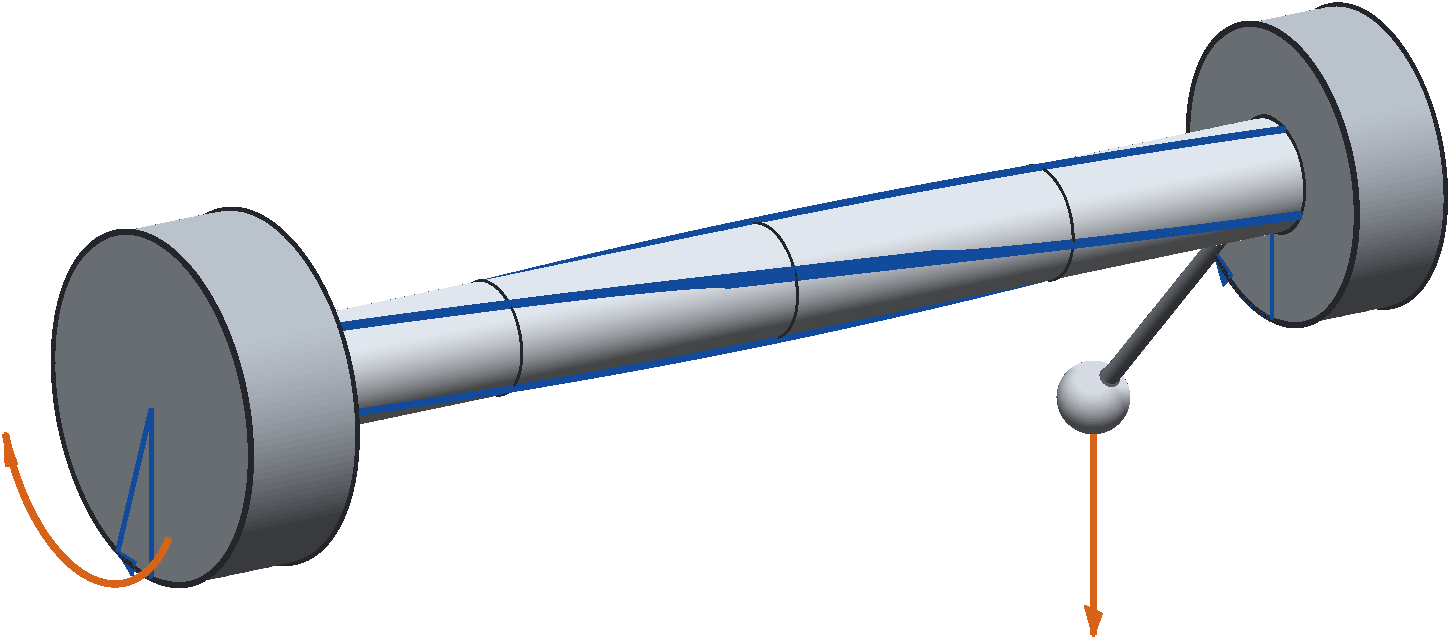}
        \put(-4,18){\color{orange!85!black}$u(t)$}
        \put(8,18){\color{blue!70!black}$y(t) = q(t,0)$}
        \put(15,0){$J_m,\ d_m$}

        \put(48,32){\color{blue!70!black}$q(t,x)$}
        \put(7,32){$x=0$}
        \put(88,45){$x=L$}

        \put(86,18){$J_l,\ d_l$}
        \put(84,39){\color{blue!70!black}$q(t,L)$}
        \put(73,20){$m$}
        \put(77,2){\color{orange!85!black}$mg$}
    \end{overpic}
    \caption{Flexible transmission system with motor torque input $u$ and
    controlled motor angle $y=q(\,\cdot\,,0)$.}
    \label{fig:flexible_transmission}
\end{figure}

Viewed in an abstract way, this model can be written as a nonlinear infinite-dimensional system, in which the nonlinearity comes from the gravitational force.

An important property of the model is its dissipativity with
respect to the supply rate $u(t)y_v(t)$, where $y_v$ denotes the
motor's angular velocity. Thus, the passive output is a velocity,
whereas the output to be tracked is the motor angle $y$.
We assume that both the motor angle and its angular velocity
are available for feedback.

This distinction between the passive velocity output and the
controlled position output is characteristic of many mechanical
systems. In particular, the two outputs are related by
$\dot y(t)=y_v(t)$, which motivates the use of position-funnel
control.

We use funnel control to prescribe the position tracking error, that is, $e(t) = y(t)-y_{\rm ref}(t)$, where $y_{\rm ref}(t)$ is some reference signal. The main idea behind funnel control is to force the tracking error to remain within some prescribed funnel boundaries for all nonnegative times. Mathematically speaking, by denoting the funnel boundaries by $\pm \varphi(t)^{-1}$ for some function $\varphi(t)$, our control objective is $\varphi(t)\Vert e(t)\Vert < 1$ for all $t\geq 0$. Although asymptotic tracking is not necessarily guaranteed, the
prescribed funnel bounds the tracking error throughout the entire
time horizon. In particular, choosing the funnel inside a prescribed
interval of admissible motor angles provides a direct mechanism
for enforcing safety constraints.

Funnel control was introduced in \cite{IlcRyaSan2002}; see
\cite{IlchmannRyan2008,BerIlcRya2021} for finite-dimensional
results and surveys. Infinite-dimensional applications include
heat equations \cite{ReiSel2015}, moving water tanks \cite{Ber2022},
nonlinear systems \cite{HasWinDoc2023}, and systems with
infinite-dimensional internal dynamics \cite{Ber2020}.
Passive systems were studied in
\cite{MarTimFel2021,GovHasPauRei2025}, and position-funnel
tracking was developed in \cite{HastirReisPosition}.

Here we extend the position-funnel approach to a flexible
transmission with nonlinear gravitational loading. The linear
shaft--motor--load dynamics are impedance passive, while the
gravitational torque is a bounded, globally Lipschitz perturbation.
We establish an energy dissipation inequality and prove the
existence of a unique global closed-loop trajectory satisfying
a strict position-funnel bound, thereby enforcing the prescribed
motor-angle safety constraint.

Section~\ref{sec:Model} describes the model, Section~\ref{sec:Pos_Tracking}
presents the controller and its analysis, and Section~\ref{sec:Simulation}
gives the numerical illustration.
\section{MODEL DESCRIPTION}\label{sec:Model}
Let $q(t,x)$ denote the torsional angle along the shaft. The PDE model that we consider is given by
\begin{equation}
    \rho \tfrac{\partial^2 q}{\partial t^2}(t,x) = \tfrac{\partial}{\partial x}\left(GJ\, \tfrac{\partial q}{\partial x}(t,x)\right),\; t\geq 0, x\in (0,L),
    \label{eq:Angle_PDE}
\end{equation}
where $\rho > 0$ is the distributed rotational inertia and $k := GJ>0$ is the torsional rigidity of the shaft. In addition, we complement \eqref{eq:Angle_PDE} with the following boundary input and boundary condition
\begin{align}
    J_m \tfrac{\partial^2 q}{\partial t^2}(t,0) &= u(t) + k \tfrac{\partial q}{\partial x}(t,0) - d_m \tfrac{\partial q}{\partial t}(t,0),
    \label{eq:Boundary_input}\\
    J_l \tfrac{\partial^2 q}{\partial t^2}(t,L) &= - k \tfrac{\partial q}{\partial x}(t,L) - d_l \tfrac{\partial q}{\partial t}(t,L)\nonumber\\&\qquad - m g l\sin(q(t,L)),
\label{eq:Boundary_Cond}
\end{align}
respectively, which model rigid inertias attached to the two end points. In \eqref{eq:Boundary_input} and \eqref{eq:Boundary_Cond}, $u$ represents the motor torque, while $J_m>0, J_l>0$, $d_m>0$ and $d_l>0$ are the motor rotational inertia, the load rotational inertia, the motor damping coefficient and the load damping coefficient, respectively. The positive constants $m, g$ and $l$ are the mass of the load, the gravitational acceleration and the distance from the load's rotation axis to its center of mass, respectively. We assume that the motor angle and its angular velocity are measured.
The position output is
\begin{equation}
    y(t) = q(t,0).
    \label{eq:pos_output}
\end{equation}
Furthermore, we consider an additional output, given by 
\[
y_v(t) = \tfrac{\partial q}{\partial t}(t,0).
\]
In what follows, we refer to $y_v$ as the velocity output. 

We represent the linear part of
\eqref{eq:Angle_PDE}--\eqref{eq:Boundary_Cond}
as a system node \cite{sta2005}, treating gravity as
a nonlinear perturbation. Introduce
\begin{align*}
    z_1(t,x) &= \tfrac{\partial q}{\partial x}(t,x),\qquad
    z_2(t,x) = \rho\tfrac{\partial q}{\partial t}(t,x),\\
    p_m(t) &= J_m\tfrac{\partial q}{\partial t}(t,0),\qquad
    p_l(t) = J_l\tfrac{\partial q}{\partial t}(t,L).
\end{align*}
We consider the Hilbert state space
$\mathcal{X}:=\mathrm{L}^2(0,L;\mathbb{R}^2)\times\mathbb{R}^2$
equipped with the inner product
\begin{align*}
    \left\langle
    \left(\begin{smallmatrix}z\\ p\end{smallmatrix}\right),
    \left(\begin{smallmatrix}\tilde z\\ \tilde p\end{smallmatrix}\right)
    \right\rangle_{\mathcal X}
    =
    \langle z,\mathcal H\tilde z\rangle_{\mathrm{L}^2(0,L;\mathbb{R}^2)}
    +
    \langle p,\mathcal J\tilde p\rangle_{\mathbb{R}^2},
\end{align*}
where
\[
    \mathcal H=
    \left(\begin{smallmatrix}
        k&0\\0&\rho^{-1}
    \end{smallmatrix}\right),
    \qquad
    \mathcal J=
    \left(\begin{smallmatrix}
        J_m^{-1}&0\\0&J_l^{-1}
    \end{smallmatrix}\right).
\]
With these definitions, the velocity output is
\[
    y_v(t)=J_m^{-1}p_m(t).
\]
Now define the operator
\[
    S=\left[\begin{smallmatrix}A\&B\\ C\&D\end{smallmatrix}\right]
\]
with $A\&B:D(S)\subset\mathcal X\times\mathbb R\to\mathcal X$ and
$C\&D:D(S)\to\mathbb R$ by
\begin{align*}
    A\&B
    \left(
    \begin{matrix}
        \left(\begin{smallmatrix}z\\p\end{smallmatrix}\right)\\ u
    \end{matrix}
    \right)
    =
    \left(
    \begin{smallmatrix}
        \frac{\mathrm d(\rho^{-1}z_2)}{\mathrm dx}\\[1mm]
        \tfrac{\mathrm d(kz_1)}{\mathrm dx}\\[1mm]
        u+(kz_1)(0)-d_mJ_m^{-1}p_m\\[1mm]
        -(kz_1)(L)-d_lJ_l^{-1}p_l
    \end{smallmatrix}
    \right)
\end{align*}
and
\begin{align*}
    C\&D
    \left(
    \begin{smallmatrix}
        \left(\begin{smallmatrix}z\\p\end{smallmatrix}\right)\\ u
    \end{smallmatrix}
    \right)
    =J_m^{-1}p_m,
\end{align*}
where $z=(z_1,z_2)^\top$ and $p=(p_m,p_l)^\top$, with
\begin{align*}
D(S)=\Bigg\{
&\left(
\begin{smallmatrix}
    \left(\begin{smallmatrix}z\\p\end{smallmatrix}\right)\\u
\end{smallmatrix}
\right)\in\mathcal X\times\mathbb R
\ \Bigg|\
\left(\begin{smallmatrix}z_1\\z_2\end{smallmatrix}\right)
\in\mathrm H^1(0,L;\mathbb R^2),\\
&(\rho^{-1}z_2)(0)=J_m^{-1}p_m,\quad
(\rho^{-1}z_2)(L)=J_l^{-1}p_l
\Bigg\}.
\end{align*}

The motor angle $y$ defined in \eqref{eq:pos_output} satisfies
\[
    \dot y(t)=y_v(t).
\]
Moreover, the load angle can be recovered from the state and the motor angle as
\[
    q(t,L)=y(t)+\int_0^L z_1(t,x)\,\mathrm dx.
\]
Let $\kappa:=mgl$ and define
\[
    \theta_l(\xi,y)
    :=
    y+\int_0^L z_1(x)\,\mathrm dx,
    \qquad
    \xi=\left(\begin{smallmatrix}z\\p\end{smallmatrix}\right)\in\mathcal X.
\]
The nonlinear perturbation $F:\mathcal X\times\mathbb R\to\mathcal X$ is
defined by
\[
    F(\xi,y)
    =
    \left(
    \begin{smallmatrix}
        0\\0\\0\\-\kappa\sin(\theta_l(\xi,y))
    \end{smallmatrix}
    \right).
\]
Hence the complete nonlinear system can be written as
\begin{subequations}\label{eq:abstract_system}
\begin{align}
    \dot\xi(t)
    &=
    A\&B
    \left(
    \begin{smallmatrix}\xi(t)\\u(t)\end{smallmatrix}
    \right)
    +F(\xi(t),y(t)),
    \label{eq:abstract_ODE_state}\\
    \dot y(t)&=y_v(t),
    \label{eq:abstract_position}\\
    y_v(t)
    &=
    C\&D
    \left(
    \begin{smallmatrix}\xi(t)\\u(t)\end{smallmatrix}
    \right).
    \label{eq:abstract_vel_out}
\end{align}
\end{subequations}
Let $A$ denote the restriction of $A\&B$ to zero input,
that is,
\[
A\xi:=A\&B
\left(\begin{smallmatrix}\xi\\0\end{smallmatrix}\right),
\qquad
D(A):=\{\xi\in\mathcal X:(\xi,0)\in D(S)\}.
\]
We show that $A$ generates a contraction semigroup.
Its domain is dense in $\mathcal X$: smooth interior fields
are dense in the distributed state space, and the endpoint
momenta can be matched by velocity fields supported in
arbitrarily small neighborhoods of the endpoints.

For $\xi=(z,p)^\top\in D(A)$, integration by parts gives
\[
\langle A\xi,\xi\rangle_{\mathcal X}
=-d_m(J_m^{-1}p_m)^2-d_l(J_l^{-1}p_l)^2\leq0.
\]
Thus, $A$ is dissipative.

To establish the surjectivity of $I-A$, let
$f=(f_1,f_2,f_m,f_l)^\top\in\mathcal X$.
Set $v:=\rho^{-1}z_2$. The equation $(I-A)\xi=f$
requires
$z_1=f_1+v'$, $z_2=\rho v$,
$p_m=J_mv(0)$, and $p_l=J_lv(L)$.
We determine $v\in\mathrm H^1(0,L)$ from
\begin{align}
&\int_0^L(\rho vw+kv'w')\,\mathrm dx
\nonumber\\
&\quad +(J_m+d_m)v(0)w(0)
 +(J_l+d_l)v(L)w(L)
\nonumber\\
&=\int_0^L(f_2w-kf_1w')\,\mathrm dx
+f_mw(0)+f_lw(L),
\label{eq:weak_resolvent}
\end{align}
for all $w\in\mathrm H^1(0,L)$.
The bilinear form on the left is bounded and coercive
on $\mathrm H^1(0,L)$, and the right-hand side
defines a bounded linear functional.
Hence the Lax--Milgram lemma yields a unique solution $v$.

Setting $z_1=f_1+v'$ gives, in the distributional sense,
\[
(kz_1)'=\rho v-f_2\in\mathrm L^2(0,L).
\]
Since $k>0$ is constant, we obtain
$z_1\in\mathrm H^1(0,L)$, and also
$z_2=\rho v\in\mathrm H^1(0,L)$.
Integration by parts in \eqref{eq:weak_resolvent}
now yields
\[
\begin{aligned}
kz_1(0)&=(J_m+d_m)v(0)-f_m,\\
kz_1(L)&=f_l-(J_l+d_l)v(L).
\end{aligned}
\]
These are precisely the endpoint equations of
$(I-A)\xi=f$. Consequently,
$\operatorname{im}(I-A)=\mathcal X$.
The Lumer--Phillips theorem shows that $A$
generates a contraction semigroup.

Moreover,
\[
A\&B\left(\begin{smallmatrix}\xi\\u\end{smallmatrix}\right)
=A\xi+Bu,\qquad
Bu=\left(\begin{smallmatrix}0\\0\\u\\0\end{smallmatrix}\right),
\]
and $C\xi=J_m^{-1}p_m$, $D=0$.
Since $B\in\mathcal L(\mathbb R,\mathcal X)$
and $C\in\mathcal L(\mathcal X,\mathbb R)$,
the operator $S$ defines a system node with
bounded input and output operators.

Furthermore, the mapping
\[
Qv:=
\left(\begin{smallmatrix}
0\\ \rho v\\ J_mv\\ J_lv
\end{smallmatrix}\right),
\qquad v\in\mathbb R,
\]
defines a bounded operator
$Q\in\mathcal L(\mathbb R,D(A))$, with
$D(A)$ equipped with its graph norm, and satisfies
$CQ=I$. In particular, $C$ is surjective.

For later use, we also record that $S$ is impedance passive. Indeed,
for $\xi=(z,p)^\top\in D(A)$ and $u\in\mathbb R$, integration by parts
and the boundary conditions give
\begin{align}
    \langle A\xi+Bu,\xi\rangle_{\mathcal X}
    &=uC\xi
      -d_m(J_m^{-1}p_m)^2
      -d_l(J_l^{-1}p_l)^2\nonumber\\
    &\leq uC\xi.
    \label{eq:linear_passivity}
\end{align}
In addition, $B=C^*$ with respect to the energy inner product on
$\mathcal X$. Thus, for any $\lambda>0$ and $u\neq0$, with
$\xi=(\lambda I-A)^{-1}Bu$, the scalar transfer function
$P(\lambda):=C(\lambda I-A)^{-1}B$ satisfies
\[
    uP(\lambda)u
    =\lambda\|\xi\|_{\mathcal X}^2
     -\langle A\xi,\xi\rangle_{\mathcal X}>0.
\]
Consequently, $P(\lambda)$ is coercive for every $\lambda>0$.

We next observe that the nonlinear system
\eqref{eq:abstract_system} is globally
well-posed. Indeed, the mapping
\[
    (\xi,y)\mapsto
    y+\int_0^L z_1(x)\,\mathrm dx
\]
is bounded and linear from $\mathcal X\times\mathbb R$ to $\mathbb R$.
More precisely, for $\xi,\widetilde\xi\in\mathcal X$
and $y,\widetilde y\in\mathbb R$,
\begin{align*}
\|F(\xi,y)\|_{\mathcal X}
&\leq \frac{\kappa}{\sqrt{J_l}},\\
\|F(\xi,y)-F(\widetilde\xi,\widetilde y)\|_{\mathcal X}
&\leq \frac{\kappa}{\sqrt{J_l}}
\left(
|y-\widetilde y|
+\sqrt{\frac Lk}\,
\|\xi-\widetilde\xi\|_{\mathcal X}
\right).
\end{align*}
Thus $F$ is globally bounded and globally Lipschitz.

On $\mathcal X_e:=\mathcal X\times\mathbb R$, define
\[
\mathcal A_e(\xi,y):=(A\xi,C\xi),
\qquad D(\mathcal A_e)=D(A)\times\mathbb R.
\]
Since $C$ is bounded, $\mathcal A_e$ generates a
$C_0$-semigroup by the bounded perturbation theorem.
Moreover, $\mathcal F(\xi,y):=(F(\xi,y),0)$ is globally
Lipschitz and bounded. Standard semilinear evolution
theory therefore yields a unique global mild solution
for every initial value and every
$u\in \mathrm{L}^1_{\rm loc}([0,\infty);\mathbb R)$;
see \cite[Ch.~6]{Pazy1983}.
We finally show that the nonlinear system
\eqref{eq:abstract_ODE_state}--\eqref{eq:abstract_vel_out} is dissipative
with respect to the supply rate $\langle u,y_v\rangle_{\mathbb R}$.
For this purpose, define the storage function
\begin{align*}
\mathcal E(\xi,y)
={}&
\tfrac12
\left[
\langle z,\mathcal Hz\rangle_{\mathrm L^2(0,L;\mathbb R^2)}
+
\langle p,\mathcal Jp\rangle_{\mathbb R^2}
\right]\\
&+
\kappa\left(1-\cos(\theta_l(\xi,y))\right),
\end{align*}
where $\xi=(z,p)^\top$. The last term is the gravitational potential energy.

We claim that
\begin{equation}
\mathcal E(\xi(t),y(t))-\mathcal E(\xi(0),y(0))
\leq
\int_0^t u(\tau)y_v(\tau)\,\mathrm d\tau.
\label{eq:Dissipativity_Ineq}
\end{equation}
It suffices to verify the corresponding differential inequality along
classical solutions.

Set $v_m:=J_m^{-1}p_m=y_v$ and
$v_l:=J_l^{-1}p_l$.
The equations for $z_1$ and $y$ imply
\begin{align*}
\tfrac{\mathrm d}{\mathrm dt}\theta_l(\xi,y)
&=\dot y+\int_0^L\dot z_1(x)\,\mathrm dx\\
&=v_m+(\rho^{-1}z_2)(L)
       -(\rho^{-1}z_2)(0)=v_l.
\end{align*}
Moreover,
$\langle F(\xi,y),\xi\rangle_{\mathcal X}
=-\kappa\sin(\theta_l(\xi,y))v_l$.
Using \eqref{eq:linear_passivity}, we obtain
along classical solutions
\begin{align*}
\tfrac{\mathrm d}{\mathrm dt}\mathcal E(\xi,y)
&=\langle A\xi+Bu+F(\xi,y),\xi\rangle_{\mathcal X}\\
&\quad+\kappa\sin(\theta_l)\dot\theta_l\\
&=uy_v-d_mv_m^2-d_lv_l^2\leq uy_v.
\end{align*}
Integration over $[0,t]$ proves
\eqref{eq:Dissipativity_Ineq}.
For general initial data in $\mathcal X_e$ and
$u\in\mathrm L^1_{\rm loc}$, approximate the data by
classical initial values and smooth inputs.
The corresponding solutions converge uniformly
on compact time intervals by the semilinear
variation-of-constants formula.
Since $C$ is bounded and the storage function is
continuous, the inequality passes to the limit.

\section{SAFETY-CONSTRAINED POSITION TRACKING}\label{sec:Pos_Tracking}

We now apply the position-funnel controller developed in
\cite{HastirReisPosition} to the nonlinear transmission system.
Let $y_{\rm ref}:\mathbb R_{\geq 0}\to\mathbb R$ be a prescribed
reference trajectory and define the tracking error
\[
    e(t):=y(t)-y_{\rm ref}(t).
\]
Since $\dot y=y_v$, its derivative is available from the measured
velocity output as
\[
    \dot e(t)=y_v(t)-\dot y_{\rm ref}(t).
\]

Let $\varphi:\mathbb R_{\geq0}\to\mathbb R_{>0}$ be a prescribed
funnel function. The control objective is
\begin{equation}
    \varphi(t)|e(t)|<1,\qquad t\geq0.
    \label{eq:position_funnel}
\end{equation}
The prescribed position funnel can directly encode safety constraints
on the motor angle. Indeed, let $y_{\min},y_{\max}:\mathbb R_{\geq0}
\to\mathbb R$ denote admissible lower and upper bounds and suppose that
the reference trajectory and the funnel function are chosen such that
\begin{equation}
    y_{\min}(t)
    \leq
    y_{\rm ref}(t)-\frac{1}{\varphi(t)}
    <
    y_{\rm ref}(t)+\frac{1}{\varphi(t)}
    \leq
    y_{\max}(t)
    \label{eq:safety_tube}
\end{equation}
for all $t\geq0$. Then \eqref{eq:position_funnel} immediately implies
\[
    y_{\min}(t)<y(t)<y_{\max}(t),
    \qquad t\geq0.
\]
Thus, the prescribed-performance funnel provides a direct mechanism
for enforcing safety constraints on the controlled motor angle.

We emphasize that the safety guarantee derived below concerns the
motor angle $y(t)=q(t,0)$. The load angle satisfies
\[
    q(t,L)=y(t)+\int_0^Lz_1(t,x)\,\mathrm dx,
\]
but is not directly constrained by the present funnel controller.
For design parameters $\alpha,\beta,\delta,\tau_f>0$, define the
recursive error
\begin{subequations}\label{eq:recursive_errors}
\begin{align}
    e_r(t)
    &:=
    \dot e(t)
    +
    \left(
        \frac{\alpha}
        {1-\varphi(t)^2e(t)^2}
        +
        \frac{\dot\varphi(t)}{\varphi(t)}
        +\delta
    \right)e(t),
    \label{eq:recursive_error}\\
    \eta(t)&:=\tau_f\varphi(t)e_r(t).
    \label{eq:normalized_recursive_error}
\end{align}
\end{subequations}
The feedback law is
\begin{equation}
    u(t)
    =
    u_{\rm ext}(t)
    -k_a(t)e_r(t)
    -
    \frac{\beta e_r(t)}
    {1-\eta(t)^2},
    \label{eq:funnel_controller}
\end{equation}
where $k_a$ is a supervised damping gain. Choose
$k_{a,0}\geq0$, $r_a\in(0,1)$ and $\Delta k,\nu>0$. As in
\cite{HastirReisPosition}, the gain is kept constant in the holding
mode. When the normalized recursive error reaches the monitoring level
$r_a$, the gain is increased with rate $\nu$ for the fixed time
$T_{\rm inc}:=\frac{\Delta k}{\nu}$
and is then held constant again. Thus, every activation increases
$k_a$ by precisely $\Delta k$. The procedure is continued recursively, with switching
times defined through the essential supremum of the
normalized recursive error, as in
\cite[Sec.~3.7]{HastirReisPosition}.

We consider the semilinear system
\begin{subequations}\label{eq:semilinear_system}
\begin{align}
\dot\xi(t)&=A\xi(t)+Bu(t)+F(\xi(t),y(t)),
\label{eq:semilinear_plant}\\
\dot y(t)&=C\xi(t).
\label{eq:integrated_output}
\end{align}
\end{subequations}
By the results of Section~\ref{sec:Model}, the linear
input--output system is impedance passive, $C$ is
surjective, and $F$ is bounded and globally Lipschitz.
We now extend the position-funnel analysis of
\cite{HastirReisPosition} to this semilinear model.
\begin{proposition}
\label{prop:nonlinear_funnel}
Assume that
$y_{\rm ref}\in \mathrm{W}^{3,\infty}(\mathbb R_{\geq0})$,
$u_{\rm ext}\in \mathrm{W}^{2,\infty}(\mathbb R_{\geq0})$,
$\varphi\in \mathrm{W}^{3,\infty}(\mathbb R_{\geq0})$, and
$\inf_{t\geq0}\varphi(t)>0$.
Let the controller parameters satisfy
\[
    \alpha,\beta,\delta,\tau_f,\Delta k,\nu>0,
    \qquad
    r_a\in(0,1),\qquad k_{a,0}\geq0.
\]
Let $\xi_0\in D(A)$ and $y_0\in\mathbb R$, and set
$e_0:=y_0-y_{\rm ref}(0)$.
Assume that
$\varphi(0)|e_0|<1$.
Define
\begin{align*}
    e_{r,0}
    &:=
    C\xi_0-\dot y_{\rm ref}(0)
    +\left(
       \frac{\alpha}{1-\varphi(0)^2e_0^2}
       +\frac{\dot\varphi(0)}{\varphi(0)}
       +\delta
     \right)e_0,\\
    \eta_0&:=\tau_f\varphi(0)e_{r,0},
\end{align*}
and assume that
\[
    |\eta_0|<r_a.
\]
Finally, define the initial input by
\[
    u_0
    :=
    u_{\rm ext}(0)-k_{a,0}e_{r,0}
    -\frac{\beta e_{r,0}}{1-\eta_0^2}.
\]
Since $D(S)=D(A)\times\mathbb R$, the condition $\xi_0\in D(A)$
implies $(\xi_0,u_0)\in D(S)$, and hence the initial compatibility
condition is satisfied.

Then the closed-loop system consisting of
\eqref{eq:semilinear_system},
\eqref{eq:recursive_errors} and \eqref{eq:funnel_controller}, and the
supervisor described above possesses a unique global trajectory with
$\xi\in \mathrm{W}^{1,\infty}_{\rm loc}(\mathbb R_{\geq0};\mathcal X)$, $y\in \mathrm{W}^{1,\infty}(\mathbb R_{\geq0})$,
$k_a\in \mathrm{W}^{1,\infty}_{\rm loc}(\mathbb R_{\geq0})$,
$u\in \mathrm{L}^\infty_{\rm loc}(\mathbb R_{\geq0})$, and
$e,\dot e,e_r,y_v\in \mathrm{L}^\infty(\mathbb R_{\geq0})$.
The evolution equations hold almost everywhere, with
$\xi(t)\in D(A)$ for almost every $t$. Furthermore,
\[
0\leq\dot k_a(t)\leq\nu\quad\text{a.e.},\qquad
0\leq k_a(t)\leq k_{a,0}+\nu t.
\]
Moreover, there exists $\varepsilon>0$ such that
\[
    \varphi(t)|e(t)|\leq1-\varepsilon,
    \qquad t\geq0,
\]
and, on every finite interval $[0,T]$, there exists
$\varepsilon_T>0$ such that
$|\eta(t)|\leq1-\varepsilon_T$
for almost every $t\in[0,T]$.
\end{proposition}
\begin{proof}
We show that the proof of \cite{HastirReisPosition} is stable under
the present gravitational perturbation. Set
\[
    \chi:=\varphi e,
    \qquad
    h_\varphi:=\frac{\dot\varphi}{\varphi}+\delta,
\]
and use the right inverse $Q$ constructed in
Section~\ref{sec:Model}. Fix $\mu>0$ and define
\begin{align*}
P(\mu)&:=C(\mu I-A)^{-1}B,\\
L_\mu&:=(\mu I-A)^{-1}B-QP(\mu),\\
G_\mu&:=AL_\mu+B.
\end{align*}
Since $CQ=I$ and $B$ is bounded, we have
$L_\mu\in\mathcal L(\mathbb R,D(A))$ and
\[
\left(
\begin{smallmatrix}L_\mu v\\v\end{smallmatrix}
\right)\in D(S),\qquad
S\left(
\begin{smallmatrix}L_\mu v\\v\end{smallmatrix}
\right)
=
\left(
\begin{smallmatrix}G_\mu v\\0\end{smallmatrix}
\right).
\]
Thus $L_\mu$ is a zero-output lifting, in the sense of
\cite[Lem.~3.3]{HastirReisPosition}.
Introduce the transformed state
\[
    \widehat z
    :=\tau_f\varphi
       \bigl(\xi-L_\mu u_{\rm ext}-Q\dot y_{\rm ref}\bigr)
       +\tau_fQh_\varphi\chi
       +\tau_f\alpha Q\gamma(\chi),
\]
where
\[
    \gamma(s):=\frac{s}{1-s^2},\qquad |s|<1.
\]
The algebraic system-node relation obtained from this transformation
is exactly the same as in \cite{HastirReisPosition}. In particular,
the gravitational term does not enter this relation, since it acts
only in the state equation and leaves the input--output relation of
the linear system node unchanged.

It remains to consider its contribution to the transformed state
equation. For $R\in(0,1)$, let $\Gamma_R$ be the smooth truncation of $\gamma$
used in \cite{HastirReisPosition}. Recall that $\Gamma_R$ is globally
Lipschitz, continuously differentiable, and that its derivative is
globally bounded and globally Lipschitz. On the
truncated problem, the inverse state transformation is
\begin{align*}
    \Xi_R(t,\widehat z,\chi)
    :={}&
    \frac{1}{\tau_f\varphi(t)}
    \bigl(
        \widehat z-\tau_fQh_\varphi(t)\chi
        -\tau_f\alpha Q\Gamma_R(\chi)
    \bigr)\\
    &+L_\mu u_{\rm ext}(t)+Q\dot y_{\rm ref}(t).
\end{align*}
Hence the gravitational perturbation contributes the term
\begin{align}
    N_R(t,\widehat z,\chi)
    :={}&
    \tau_f\varphi(t)
    F\left(
        \Xi_R(t,\widehat z,\chi),
        y_{\rm ref}(t)+\varphi(t)^{-1}\chi
    \right)
    \label{eq:transformed_perturbation}
\end{align}
to the transformed state equation.

Let $w=(\widehat z,\chi,k_a)$ and define
\[
\mathcal N_R(t,w):=
\begin{pmatrix}
N_R(t,\widehat z,\chi)\\0\\0
\end{pmatrix}.
\]
Since the first distributed component of $Qv$
vanishes for every $v\in\mathbb R$, the load angle
appearing in \eqref{eq:transformed_perturbation}
reduces to
\begin{align*}
\Theta(t,\widehat z,\chi)
={}&y_{\rm ref}(t)+\frac{\chi}{\varphi(t)}+\frac{1}{\tau_f\varphi(t)}
\int_0^L\widehat z_1(x)\,\mathrm dx\\
&\quad+\int_0^L(L_\mu u_{\rm ext}(t))_1(x)\,\mathrm dx.
\end{align*}
Consequently,
\[
N_R(t,\widehat z,\chi)
=\tau_f\varphi(t)
\begin{pmatrix}
0\\0\\0\\-\kappa\sin\Theta(t,\widehat z,\chi)
\end{pmatrix}.
\]
In particular, this perturbation is independent
of the truncation $\Gamma_R$.

Using the Lipschitz continuity of the sine function
and the energy norm, we obtain
\begin{align*}
&\|\mathcal N_R(t,w)-\mathcal N_R(t,\widetilde w)\|\\
&\quad\leq\frac{\kappa}{\sqrt{J_l}}
\left(
\sqrt{\frac Lk}
\|\widehat z-\widetilde{\widehat z}\|_{\mathcal X}
+\tau_f|\chi-\widetilde\chi|
\right).
\end{align*}
Thus $\mathcal N_R$ is globally Lipschitz in the
state, uniformly in time. Moreover,
\[
\|\mathcal N_R(t,w)\|
\leq\frac{\tau_f\|\varphi\|_\infty\kappa}{\sqrt{J_l}}.
\]
The regularity assumptions on the prescribed data
also imply that the time dependence is Lipschitz
continuous uniformly on bounded state sets.

Decompose
\[
\mathcal N_R(t,w)
=\overline{\mathcal N}_R(t,w)+n_R(t),
\]
where
\[
\begin{aligned}
\overline{\mathcal N}_R(t,w)
&:=\mathcal N_R(t,w)-\mathcal N_R(t,0),\\
n_R(t)&:=\mathcal N_R(t,0).
\end{aligned}
\]
Then $\overline{\mathcal N}_R(t,0)=0$, and
$n_R$ belongs to $\mathrm{W}^{1,\infty}$ with values
in the transformed product space.

The maximally monotone principal operator of the
truncated evolution problem in
\cite{HastirReisPosition} remains unchanged.
Only its globally Lipschitz perturbation is
augmented by $\overline{\mathcal N}_R$, while
$n_R$ is added to the prescribed forcing term.
The resulting terms satisfy the same regularity
requirements as in
\cite[Prop.~3.6]{HastirReisPosition}.
Hence the truncated evolution argument applies
in both fixed supervisor modes.
The pointwise domain property and the continuation
of the fixed-mode trajectories follow by the
arguments of
\cite[Lem.~3.7 and Prop.~3.11]{HastirReisPosition}.
The a priori position-funnel estimate is unchanged as well. Indeed,
from the definition of the recursive error one obtains exactly
\[
    \dot\chi
    =
    \tau_f^{-1}\eta
    -\alpha\frac{\chi}{1-\chi^2}
    -\delta\chi.
\]
This relation is independent of the plant dynamics and yields the same
uniform separation of $\chi$ from the boundary of $(-1,1)$ as in
\cite{HastirReisPosition}.

Finally, the algebraic system-node relation between $\widehat z$, the
normalized recursive error $\eta$, and the feedback input is unchanged.
In particular, the gravitational perturbation does not occur in the
resolvent identity used to estimate $\eta$. On every finite time
interval, the truncated evolution yields
$\widehat z\in \mathrm{W}^{1,\infty}$,
$\widehat f\in \mathrm{L}^\infty$,
where $\widehat f$ denotes the state component in the corresponding
system-node relation. Hence the right-hand side of the resolvent
identity is bounded exactly as in \cite{HastirReisPosition}.
Impedance passivity and coercivity of the transfer operator therefore
give the same estimate for $\gamma(\eta)$ and yield a strict
finite-horizon margin
$|\eta(t)|\leq1-\varepsilon_T$
for almost every $t\in[0,T]$.

The uniform position-funnel margin and the strict
finite-horizon recursive-error margin allow us to
remove the truncation by the first-exit argument
of \cite{HastirReisPosition}.
The pointwise domain property permits restarting
the evolution at every supervisor switch.
Since every gain-increase phase has the fixed
duration $T_{\rm inc}>0$, switching times cannot
accumulate in finite time. The recursive supervisor
construction therefore yields a unique global
closed-loop trajectory.

It remains to establish the asserted global bounds.
The position-funnel estimate and
$\inf_{t\geq0}\varphi(t)>0$ imply
$e\in \mathrm{L}^\infty(\mathbb R_{\geq0})$.
Moreover, $|\eta(t)|<1$ almost everywhere gives
\[
|e_r(t)|\leq
\frac{1}{\tau_f\inf_{s\geq0}\varphi(s)}.
\]
Since $\dot\varphi/\varphi$ is bounded and the
position error stays uniformly separated from
the funnel boundary, \eqref{eq:recursive_error}
implies $\dot e\in \mathrm{L}^\infty(\mathbb R_{\geq0})$.
Consequently,
\[
y_v=\dot e+\dot y_{\rm ref}
\in \mathrm{L}^\infty(\mathbb R_{\geq0}),
\]
and $y=y_{\rm ref}+e$ belongs to
$\mathrm{W}^{1,\infty}(\mathbb R_{\geq0})$.
Finally, the supervisor satisfies
$\dot k_a\in\{0,\nu\}$ almost everywhere,
which yields the stated gain bounds.
\end{proof}

\section{NUMERICAL SIMULATION}\label{sec:Simulation}

We illustrate the closed-loop behavior of the flexible
transmission system over a time horizon of $10$ seconds.
The shaft is discretized using $N=20$ continuous
piecewise-linear finite elements. The resulting
semidiscrete model takes the form
\[
M\ddot{\boldsymbol q}+D\dot{\boldsymbol q}
+K\boldsymbol q+\kappa\sin(q_N)\boldsymbol e_N
=u\boldsymbol e_0,
\]
where $M$, $D$, and $K$ denote the mass, damping, and
stiffness matrices, respectively. The endpoint inertias
and damping are incorporated into $M$ and $D$.
The motor and load angles are $q_0$ and $q_N$.
The discretization preserves the mechanical energy
balance of the original system.

We choose $L=\rho=1$, $k=20$, $J_m=0.3$,
$J_l=0.5$, $d_m=1$, $d_l=0.75$, and $\kappa=3$.
The reference trajectory is obtained by piecewise quintic
interpolation of the values
$0,1,-1,0.7,-0.6,0.9,-0.4,0,0$ at the times
$0,1.3,2.8,4,5.4,6.8,8.3,9.5,10$,
respectively, with vanishing first and second
derivatives at the interpolation points.
The funnel function is prescribed by
\[
\varphi(t)^{-1}=0.3+0.6e^{-t/1.5}.
\]
The reference and funnel are chosen such that
$|y_{\rm ref}(t)|+\varphi(t)^{-1}<1.6$,
which enforces the motor-angle safety constraint
$|y(t)|<1.6$.

The controller parameters are
$\alpha=\beta=1$, $\delta=0.5$, $\tau_f=0.25$,
$k_{a,0}=0$, $r_a=0.8$, $\Delta k=2$, and $\nu=10$,
with $u_{\rm ext}=0$.
Starting from rest, we integrate the semidiscrete system
using MATLAB's \texttt{ode15s}, with relative and absolute
tolerances $10^{-7}$ and $10^{-9}$ and a maximum time
step of $10^{-2}$. Supervisor switching times are
determined by event detection.

\begin{figure}[t]
    \centering
    \includegraphics[width=0.96\columnwidth]
    {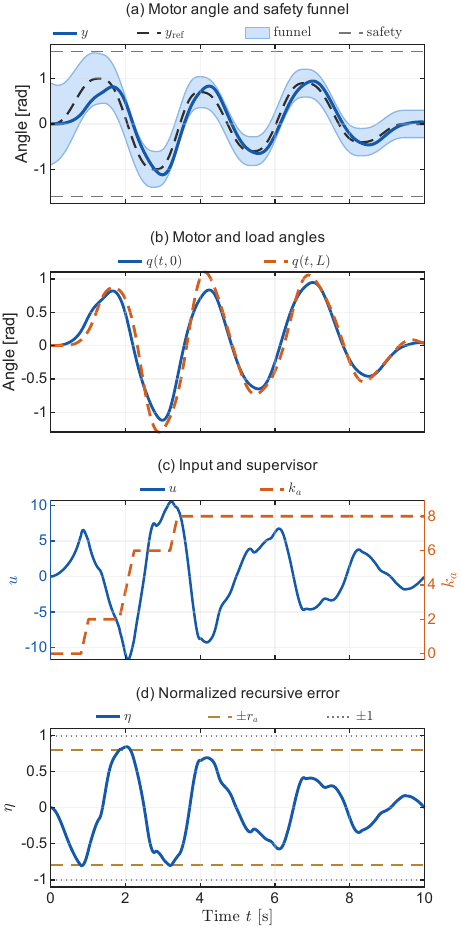}
    \caption{Closed-loop simulation:
(a) motor angle, reference, shaded position funnel,
and safety limits; (b) motor and load angles;
(c) motor torque and supervised gain;
(d) normalized recursive error with monitoring levels
$\pm r_a$ and funnel limits $\pm1$.}
    \label{fig:transmission_simulation}
\end{figure}

As shown in Fig.~\ref{fig:transmission_simulation},
the motor angle remains strictly within the prescribed
position funnel. The maximum normalized position error
is $0.834$, while the normalized recursive error
satisfies $\max_t|\eta(t)|=0.845$.
The supervisor is activated four times, resulting
in a final gain of $k_a=8$.
The maximum applied motor torque is $11.69$.
The load angle exhibits additional motion due to
the shaft flexibility, with a maximum motor-to-load
angular difference of $0.334$ rad.

\section{CONCLUSION}\label{sec:Ccl}

We studied safety-constrained motor-angle tracking for a flexible
transmission system with a gravitationally loaded inertia. The
coupled PDE--ODE model was formulated as a semilinear system with an
impedance-passive linear part and a bounded, globally Lipschitz
perturbation due to gravity. We established an energy dissipation
inequality for the nonlinear system and showed that supervised
position-funnel feedback yields a unique global closed-loop
trajectory, preserves a uniform position-funnel margin, and keeps
the velocity tracking error globally bounded. When
the position funnel is contained in an admissible motor-angle interval,
this yields a safety guarantee for the motor angle. The present
controller does not, however, directly constrain the load angle.




\bibliographystyle{IEEEtran} 
\bibliography{References}

\end{document}